\documentclass[review]{elsarticle}

\usepackage{lineno,hyperref}
\journal{Journal of \LaTeX\ Templates} 

\usepackage[utf8]{inputenc} 
\usepackage[T1]{fontenc}
\usepackage{url}
\usepackage{ifthen}
\usepackage{xfrac}
\usepackage{amssymb}
\usepackage{epsfig}
\usepackage{amsfonts}
\usepackage{pict2e}
\usepackage{color}
\usepackage{amsmath}
\usepackage{graphicx}
\usepackage{mathtools, amsthm, amsfonts, amssymb}
\usepackage{float}
\usepackage{tikz}
\tikzset{>=latex}
\usepackage{caption}
\usepackage{subcaption}
\usepackage{booktabs}

\usepackage{commath}
\usepackage{bbm}
\usepackage{bm}
\usepackage{accents}

\DeclareMathAccent{\wtilde}{\mathord}{largesymbols}{"65}

\newcommand{\NN}{\mathbb{N}}

\newcommand{\RR}{\mathbb{R}}

\newcommand{\E}{\mathbb{E}}

\newcommand{\hx}{\hat{x}}

\usepackage{xcolor}

\usepackage{hhline}

\let\leqx\leqslant

\newcommand{\doasympleqx}{%
	\hbox{\ooalign{%
			\noalign{\kern.25ex}
			$\leqslant$\cr
			\noalign{\kern1.25ex}
			\smash{$\sim$}\cr
	}}%
}

\newcommand{\doasympdomleq}{%
	\hbox{\ooalign{%
			\noalign{\kern.25ex}
			$\preccurlyeq$\cr
			\noalign{\kern1.25ex}
			\smash{$\sim$}\cr
	}}%
}

\newcommand{\doasympasympdomleq}{%
	\hbox{\ooalign{%
			\noalign{\kern.25ex}
			$\preccurlyeq$\cr
			\noalign{\kern1.25ex}
			\smash{$\sim$}\cr
			\noalign{\kern0.5ex}
			\smash{$\sim$}\cr
	}}%
}

\let\leqx\leqslant

\newcommand{\doasympgeqx}{%
	\hbox{\ooalign{%
			\noalign{\kern.25ex}
			$\geqslant$\cr
			\noalign{\kern1.25ex}
			\smash{$\sim$}\cr
	}}%
}

\newcommand{\be}{\begin{equation}}
\newcommand{\ee}{\end{equation}}

\newcommand{\doasympasympleqx}{%
	\hbox{\ooalign{%
			\noalign{\kern.25ex}
			$\leqx$\cr
			\noalign{\kern1.25ex}
			\smash{$\sim$}\cr
			\noalign{\kern0.5ex}
			\smash{$\sim$}\cr
	}}%
}

\newtheorem{Theorem}{Theorem}
\newtheorem{Definition}[Theorem]{Definition}
\newtheorem{Lemma}[Theorem]{Lemma}

\newtheorem{Remark}[Theorem]{Remark}

\newtheorem{Assumptions}[Theorem]{Assumptions}

\def\C{{\mathcal C}}
\def\D{{\mathcal D}}
\def\E{{\mathcal E}}

\def\G{{\mathcal G}}

\def\M{{\mathcal M}}

\def\P{{\mathcal P}}

\def\S{{\mathcal S}}

\def\U{{\mathcal U}}

\def\X{{\mathcal X}}
\def\Y{{\mathcal Y}}

\def\NN{{\mathbb N}}

\def\RR{{\mathbb R}}

\def\CH{{\mathcal CH}(\X,\Y)}

\def\CHc{{\mathcal CH}_c(\X,\Y)}

\def\mA{\bm{\mathrm{A}}}
\def\mC{\bm{\mathrm{C}}}

\usepackage{xcolor,colortbl}

\newcommand{\rem}[1]{}

\vfuzz \hfuzz
\begin{document}
\begin{frontmatter}

\title{On the Semi-Decidability of Remote State Estimation and Stabilization via Noisy Communication Channels}

\author{Holger Boche}
\address{Chair of Theoretical Information Technology\\ 
	                  Technical University of Munich\\
                    D-80333 Munich, Germany\\
                    and\\
                    Munich Center for Quantum Science\\ and Technology (MCQST)\\ D-80799 Munich, Germany\\
                    Email: boche@tum.de}

\author{Yannik B\"ock}
\address{Theoretical Information Technology\\ 
	                  Technical University of Munich\\
                    D-80333 Munich, Germany\\
                    Email: yannik.boeck@tum.de}   
										
\author{Christian Deppe}
\address{Institute for Communications Engineering\\ 
	Technical University of Munich\\
	D-80333 Munich, Germany\\
	Email: christian.deppe@tum.de}
																				


\begin{abstract}

We consider the task of remote state estimation and stabilization of disturbed linear plants via noisy communication channels. In 2007 Matveev and Savkin established a surprising link between this problem and Shannon's theory of zero-error communication. By applying very recent results of computability of the channel reliability function and computability of the zero-error capacity of noisy channels by Boche and Deppe, we analyze if, on the set of linear time-invariant systems paired with a noisy communication channel, it is uniformly decidable by means of a Turing machine whether remote state estimation and stabilization is possible. The answer to this question largely depends on whether the plant is disturbed by random noise or not. Our analysis incorporates scenarios both with and without channel feedback, as well as a weakened form of state estimation and stabilization. In the broadest sense, our results yield a fundamental limit to the capabilities of computer-aided design and autonomous systems, assuming they are based on real-world digital computers.
\end{abstract}


\end{frontmatter}

\section{Introduction}
We consider the problem of remote state estimation and stabilization of an unstable, linear, time-invariant plant via a noisy communication channel \cite{MS07}. Schematically, the set-up is depicted in Figure \ref{fig:GeneralSchematics}. The status of an unstable, linear plant is observed by a local sensor. Subsequently, the sensor data is feed into an encoder that prepares the data for transmission trough a discrete, memoryless channel. Based on the channel output, the remote decoder/estimator tries to estimate the current status of the plant. Depending on the specific scenario, a feedback-link to enhance the transmission quality and/or a control-link to stabilize the plant may additionally be available.

\begin{figure}[htbp!]\linespread{1}
    \centering
    \begin{tikzpicture}
            \draw[dashed, very thick, ->, shorten >=1.5pt] (-4,1.25) -- (-3,1.25); 
            \draw[dashed, very thick, ->, shorten >=1.5pt] (-4.8,1.25) -- (-4,1.25); 
            \draw (-4,1.25)  node[fill = white, anchor=center, align=center]
            {\(\zeta(n)\)};
            \draw[ultra thick, fill=black!20] (-3,1.5) rectangle (-1,0.5);
            \draw (-2,1)  node[anchor=center] {\(\mA\)};
            \draw (-2,1.5)  node[anchor=south] {Plant};
            \draw[very thick, ->, shorten >=1.5pt, shorten <=1.5pt] (-1,1) -- (1,1); 
            \draw (0,1)  node[fill = white, anchor=center]
            {\(x(n)\)};
            \draw[ultra thick, fill=black!20] (3,1.5) rectangle (1,0.5);
            \draw (2,1)  node[anchor=center] {\(\mC\)};
            \draw (2,1.5)  node[anchor=south] {Sensor};
            
            \draw[very thick, <-, shorten >=1.5pt, shorten <=1.5pt] (4,-1) -- (4,1) -- (3,1); 
            \draw (4,0)  node[fill = white, anchor=center, align=center]
            {\(y(n)\)};
            \draw[ultra thick, fill=black!20] (4.5,-2) rectangle (2.5,-1);
            \draw (3.5,-1.5)  node[anchor=center] {\(\E\)};
            \draw (3.5,-2)  node[anchor=north] {Encoder};
            \draw[very thick, ->, shorten >=1.5pt, shorten <=1.5pt] (2.5,-1.5) -- (0.7,-1.5); 
            \draw (1.65,-1.5)  node[fill = white, anchor=center, align=center]
            {\(e(n)\)};
            \fill[fill = black!20] (-1,-1.5) -- (-0.7, -1.1) -- (1,-1.1) -- (0.7,-1.5) -- (1,-1.9) -- (-0.7,-1.9);
            \draw (0,-1.5)  node[anchor=center] {\(W\)};
            \draw (0,-2)  node[anchor=north] {DMC};
            \draw[very thick, ->, shorten >=1.5pt, shorten <=1.5pt] (-1,-1.5) -- (-2.5,-1.5); 
            \draw (-1.65,-1.5)  node[fill = white, anchor=center, align=center]
            {\(s(n)\)};
            \draw[ultra thick, fill=black!20] (-4.5,-2) rectangle (-2.5,-1);
            \draw (-3.5,-1.5)  node[anchor=center] {\(\D\)};
            \draw (-3.5,-2)  node[anchor=north] {Decoder/Estimator};
            \draw[dashed, very thick, ->, shorten >=1.5pt, shorten <=1.5pt] (-4,-1) -- (-4,0.75) -- (-3,0.75); 
            \draw (-4,0)  node[fill = white, anchor=center, align=center]
            {Control-\\link};
            \draw[dashed, very thick, ->, shorten >=1.5pt, shorten <=1.5pt] (-3,-1) -- (-3,-0.5) -- (3,-0.5) -- (3,-1); 
            \draw (0,-0.5)  node[fill = white, anchor=center, align=center]
            {Feedback-link};
            \draw[dashed, very thick, ->, shorten >=1.5pt, shorten <=1.5pt] (-4.5,-1.5) -- (-5.5,-1.5) -- (-5.5,1);
            \draw (-5.5,0)  node[fill = white, anchor=center, align=center]
            {\(\hat{x}(n)\)};
    \end{tikzpicture}
    \caption{General schematics of the remote state estimation and stabilization problem. Different scenarios investigated in this work are indicated by dashed lines.}
    \label{fig:GeneralSchematics}
\end{figure}
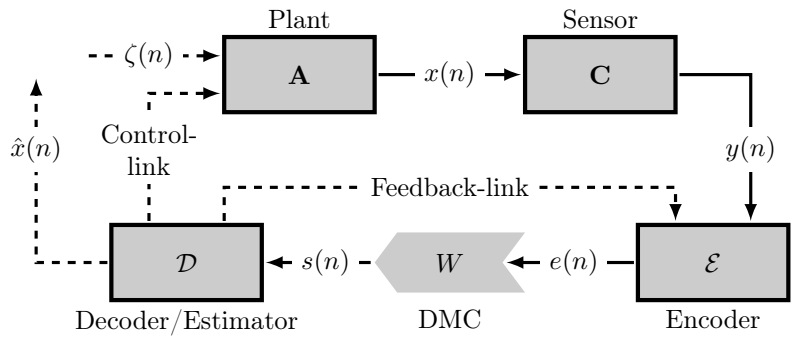




\rem{The problems of state estimation and stabilization of unstable linear plants are of great importance in control theory. 
We consider the problem of state estimation
for a discrete linear partially observed time invariant system (see \cite{MS07}).
We assume that the estimate is required at a remote
location, and the sensor data are transmitted to the estimator
via a finite capacity digital communication channel. 
We consider a discrete memoryless channel.}

The strong objective takes the form of almost
sure stability/observability: the error must be kept small
along almost all possible trajectories, i.e., with probability one.
Problems of this kind play an essential role in control theory \cite{BL00, EM01, HOV02, IF02, HT05, HF02, L03, MS04, MS05, MS05a, MS05b, PS01, S06, SP03, TM04, TM04b, WB99}.
In all the cited papers it is assumed that either the communication channels
or the plants are noiseless. However, this means that we are neglecting any uncertainty of the system. In \cite{MS07} the more realistic situation is considered that the plant and channel are noisy. 
A remote control system is examined more precisely where the sensor signals are transmitted over a noisy channel to the controller observer. 
As channel model, they use the discrete memoryless channel (DMC) introduced by Shannon in \cite{S48}. The linear plant
is affected by additive stochastic disturbances, which are
independent of the channel noise. 

In \cite{MS07} a description of this model is given. 
They show that the capability of the noisy channel to
ensure almost sure stabilizability and observability of the
plant is identical to its capability to transmit information
with the zero probability of error. 
This requires the zero-error capacity that Shannon introduced in \cite{S56}. In the paper \cite{BD20Z} the computability of the zero-error capacity is considered.

In \cite{MS07} it is shown that the
boundary of the almost sure stabilizability and observability
domain is given by the channel zero error capacity.
Furthermore, in the
face of channel errors, external disturbances obeying
a common and arbitrarily small deterministic bound at
any sample and time unavoidably accumulate and cause,
sooner or later, arbitrarily large stabilization and estimation
errors. If the zero-error capacity of a DMC is zero, 
an asymptotically
unstable linear plant can be neither stabilized nor
observed with a bounded error over such a DMC.

\rem{Both the state estimation problem and the plant stabilization problem are modeled in \cite{MS07} using a noisy channel, the DMC $W$. We consider the plant stabilization problem with and without feedback. For both problems we have the dynamic system, i.e. the plant is described by dynamic parameters that completely determine the system behavior and the noisy channel is described by communication parameters that completely determine the channel.}\rem{In this paper we want to find out whether the two computable problems can be solved algorithmically. As a computability model, we consider Turing machines with Turing computability. All algorithms that can be computed with computers are Turing computable.}

\rem{We want to answer whether a Turing machine $TM_{SE}$ exists with only one output state "stop", which receives the dynamic parameters of the plan as input and those of the channel, which stops exactly when the state estimation problem can be solved. Furthermore, we consider the question of whether there is a Turing machine that outputs for the problem of stability whether the problem is solvable or not.}

In the scope of this work, we investigate the remote state estimation and stabilization problem with regards to algorithmic decidability. In particular, we analyze if there exists Turing machines that, given the parameters of the state estimation an stabilization set-up as an input, reach their halting state in a finite number of steps if and only if remote state estimation and stabilization is possible.
In particular, the Turing machine is thus not required to halt for every input. Any algorithm that can be executed by a real-world digital computer can, in theory, be simulated by a Turing Machine.

If the remote state estimation and stabilization problem is undecidable on a Turing Machine, it is certainly not decidable on any real-word digital computer,
which poses a severe drawback for computer-aided design and autonomous systems. Computer-aided design tools commonly feature an \emph{exit flag} functionality that serves as an indicator of success upon termination of the automated design process. Depending on the exit flag, the return value of the executed algorithm is marked to be either a valid solution for the design problem or not. In our context, a hypothetical automated design tool would be presented with description \((\mA,\mC,W)\) of the set-up depicted in Figure
\ref{fig:GeneralSchematics}, and return a triple \(\E,\D, f\) consisting of an encoder, a decoder/estimator and a boolean exit flag \(f\). If the exit flag equals \emph{true}, then the pair \((\E,\D)\) is a valid solution for the state estimation problem characterized by \((\mA,\mC,W)\). If the exit flag equals \emph{false}, the algorithm was unable to solve the state estimation problem 
and the pair \((\E,\D)\) is some arbitrary dummy value.
If, however, the state estimation problem is undecidable on a Turing machine, an automated design tool that satisfies this requirement cannot exist

A similar problem arises for scenarios where autonomous adaption to changes in the system dynamics are required, as it is the case in wireless communication. Consider therefore a mobile dynamic system, e.g., a satellite or a drone, that is accessible to a base station via a wireless data link. While the parameters of the mobile system may be known in advance, the characteristics of the wireless transmission is likely to change over time. Assuming the channel remains constant during each data frame, the mobile system and the base station have to agree on a new pair \((\E,\D)\) for each transmission. If the state estimation problem is not solvable in an automated way, this task can not be accomplished.

\begin{figure}[htbp!]
    \centering
            \begin{tikzpicture}\linespread{1}
            \draw[ultra thick, fill=black!20] (-4.5,2) rectangle (4.5,-2);
            \draw[ultra thick] (-4.5,0) -- (4.5,0);
            \draw[ultra thick] (0,0) -- (0,2); 
            \draw (-2.25,1)  node[align = left, anchor=center]  {\(\bullet\) \textit{Section \ref{State}}\\
                \(\bullet\) Solvable: If \(\eta < C_0^{FB}\)\\
                \(\bullet\) \textbf{Semi-decidable: No}};
            \draw (2.25,1)  node[align = left, anchor=center]  {\(\bullet\) \textit{Section \ref{State}}\\
                \(\bullet\) Solvable: If \(\eta < C_0^{\phantom{FB}}\)\\
                \(\bullet\) \textbf{Semi-decidable: No}};
            \draw (0,-1)  node[align = left, anchor=center]  {\(\bullet\) \textit{Section \ref{Stab}}\\
                 \(\bullet\) Possible: If \(\eta < C_0^{FB}\)\\
                 \(\bullet\) \textbf{Semi-decidable: No}};
            \draw (-4.5,-2)  node[align = right, anchor=east]  {\small \textit{Control-}\\
                \small \textit{link}};
            \draw (4.5,-2)  node[align = left, anchor=west]  {\small \textit{Plant dis-}\\
                \small\textit{turbance}};
            \draw (-4.5,-1)  node[align = right, anchor=east]  {\small Yes};
            \draw (4.5,1)  node[align = left, anchor=west]  {\small Yes};
            \draw (4.5,-1)  node[align = left, anchor=west]  {\small Yes};
            \draw (-4.5,1)  node[align = right, anchor=east]  {\small No};
            \draw (0,2)  node[align = center, anchor=south]  {\small\textit{Feedback-}\\
                \small \textit{link}};
            \draw (-2,2)  node[align = center, anchor=south]  {\small Yes};
            \draw (2,2)  node[align = center, anchor=south]  {\small No};
            \begin{scope}[yshift = -3cm]
                \draw[ultra thick, fill=black!20] (-4.5,1) rectangle (4.5,-1);
                \draw (0,0)  node[align = left, anchor=center]  {\(\bullet\) \textit{Section \ref{Weak}}\\
                    \(\bullet\) Solvable: If \(\eta < C^{\phantom{FB}}\)\\
                    \(\bullet\) \textbf{Semi-decidable: Yes}};
                \draw (-4.5,0)  node[align = right, anchor=east]  {\small Yes/ \\ \small No};
                \draw (4.5,0)  node[align = left, anchor=west]  {\small No};
            \end{scope}
            \end{tikzpicture}
        \caption{Schematic outline of the paper. Results of this work are highlighted in bold font.}
    \label{fig:Overview}
\end{figure}

\begin{figure}[htbp!]
    \centering
            \begin{tikzpicture}\linespread{1}
            \draw[ultra thick, fill=black!20] (-4.5,2) rectangle (4.5,-2);
            \draw[ultra thick] (-4.5,0) -- (4.5,0);
            \draw[ultra thick] (0,0) -- (0,2); 
            \draw (-2.25,1)  node[align = left, anchor=center]  {\(\bullet\) \textit{Section \ref{State}}\\
                \(\bullet\) Unsolvable: If \(C_0^{FB} < \eta\)\\
                \(\bullet\) \textbf{Semi-decidable: Yes}};
            \draw (2.25,1)  node[align = left, anchor=center]  {\(\bullet\) \textit{Section \ref{State}}\\
                \(\bullet\) Unsolvable: If \(C_0 < \eta\)\\
                \(\bullet\) Semi-decidable: Open};
            \draw (0,-1)  node[align = left, anchor=center]  {\(\bullet\) \textit{Section \ref{Stab}}\\
                 \(\bullet\) Impossible: If \(C_0^{FB} < \eta  \)\\
                 \(\bullet\) \textbf{Semi-decidable: Yes}};
            \draw (-4.5,-2)  node[align = right, anchor=east]  {\small \textit{Control-}\\
                \small \textit{link}};
            \draw (4.5,-2)  node[align = left, anchor=west]  {\small \textit{Plant dis-}\\
                \small\textit{turbance}};
            \draw (-4.5,-1)  node[align = right, anchor=east]  {\small Yes};
            \draw (4.5,1)  node[align = left, anchor=west]  {\small Yes};
            \draw (4.5,-1)  node[align = left, anchor=west]  {\small Yes};
            \draw (-4.5,1)  node[align = right, anchor=east]  {\small No};
            \draw (0,2)  node[align = center, anchor=south]  {\small\textit{Feedback-}\\
                \small \textit{link}};
            \draw (-2,2)  node[align = center, anchor=south]  {\small Yes};
            \draw (2,2)  node[align = center, anchor=south]  {\small No};
            \begin{scope}[yshift = -3cm]
                \draw[ultra thick, fill=black!20] (-4.5,1) rectangle (4.5,-1);
                \draw (0,0)  node[align = left, anchor=center]  {\(\bullet\) \textit{Section \ref{Weak}}\\
                    \(\bullet\) Unsolvable: If \(C < \eta \)\\
                    \(\bullet\) \textbf{Semi-decidable: Yes}};
                \draw (-4.5,0)  node[align = right, anchor=east]  {\small Yes/ \\ \small No};
                \draw (4.5,0)  node[align = left, anchor=west]  {\small No};
            \end{scope}
            \end{tikzpicture}
        \caption{Schematic outline of the paper. Results of this work are highlighted in bold font.}
    \label{fig:OverviewVer2}
\end{figure}

Applying the results from \cite{BD20Z} and \cite{BD20R}, we investigate different variants of the remote state estimation and stabilization problem with respect to decidability, as visualized in Figure \ref{fig:Overview} and Figure \ref{fig:OverviewVer2}. The outline of the may be summarized as follows:
In Section \ref{Basic}, we discuss the preliminaries on \emph{computability theory}, \emph{channel coding} and \emph{control theory} which will be applied subsequently. In particular, we introduce a formal mathematical description for the
    set-up depicted in Figure \ref{fig:GeneralSchematics}.
    Section \ref{State} discusses the remote state estimation problem for disturbed plants in the absence of a control-link. The state estimation is required to satisfy the objective of almost sure observability. We discuss scenarios both with and without feedback-link, follwoing Matveev and Savkin \cite{MS07} as well as Liberzon \cite{L03,L03a}.
    In section \ref{Stab}, we discuss the remote stabilization problem for disturbed plants. Again, we discuss scenarios both with and without feedback-link.
    In section \ref{Weak}, we discuss the remote state estimation and stabilization problem for undisturbed plants, again both with and without feedback-link.
    The paper concludes in Section \ref{Conclusions} with a summary and subsumption of our work.

\begin{Remark}
    The concepts of \emph{decidability} and \emph{semi-decidability} will be formally introduced in Section \ref{Basic}. Throughout this work, we prove results concerning the semi-decidability of
    of specific classes of pairs \((\mA,W)\). Strictly speaking, our results are stronger than the results depicted in Figure \ref{fig:Overview}, since the non-semi-decidability of a set immediately implies its non-decidability.
\end{Remark}

\section{Basic definitions and results}\label{Basic}
In this section we will introduce the basic definitions and results that we will need for our investigations. We want to examine a problem from the control problem for its algorithmic solvability. Therefore, we start with the basics of computability theory and continue with the basics of channel coding and zero-error theory. These are required in the control theoretical model that we introduce afterwards.

\subsection{Basics of computational theory}

In this section we formally introduce the concepts of computability (Turing, \cite{T36}), effective analysis, and the concepts of the zero-error capacity (Shannon, \cite{S56}).

A partial function from $\NN$ to $\NN$ is called partial recursive if it can be computed by a Turing machine; that is, if there exists a Turing machine that accepts input $x$ exactly if $f(x)$ is defined, and, upon acceptance, leaves the string $f(x)$ on its tape \cite{Soa87}.

We would like to make statements about the computability of the zero-error capacity. This capacity is generally a real number. 
Therefore, we first define when a real number is computable. For this we need the following two definitions.

We use the concepts of recursive functions (see \cite{Go30,Go34,Kle52, Min61}) 
and computable numbers (see \cite{PoRi17,W00}).

\begin{Definition}\label{ber}
A sequence of rational numbers $\{r_n\}_{n\in\NN}$ is called a computable sequence if  there  exist  partial recursive  functions $a,b,s:\NN\to\NN$ 
with $b(n)\not = 0$ for all $n\in\NN$ and 
\[
r_n= (-1)^{s(n)}\frac {a(n)}{b(n)},\ \    n\in\NN. 
\]
\end{Definition}
\begin{Definition}\label{compreal}
A real number $x$ is said to be computable if there exists a computable sequence of rational numbers $\{r_n\}_{n\in\NN}$ such that $|x-r_n|<2^{-n}$ 
for all $n\in\NN$. We denote the set of computable real numbers by $\RR_c$.
\end{Definition}

\begin{Definition}
 	A sequence of functions $\{F_n\}_{n\in\NN}$ with $F_n:\X\to \mathbb{R}_{c}$ is computable if the mapping $(i,x)\to F_i(x)$ is computable. 
 \end{Definition}
 
 \begin{Definition}
	A computable sequence of computable functions  $\{F_N\}_{N\in\NN}$ is called computably convergent to $F$ if there exists a partial recursive 
	function $\phi:\NN\times X\to \NN$, such that
	\[
	\left| F(x)-F_N(x)\right|<\frac 1{2^M}
	\]
	holds true for all $M\in\NN$, all $N\geq\phi(M,x)$ and all $x\in X$.
	\end{Definition}
\begin{Remark}
A number $x\in\RR_c$ is computable if and only if there is a computable sequence $\{r_n\}_{n\in\NN}$ of rational numbers and a partial recursive function $\phi:\NN\to\NN$, such that 
\[
|x-r_n|<\frac 1{2^M}
\]
holds true for all $M\in\NN$ and all \(n \geq \phi(M)\). 
\end{Remark}

For further investigations, we also need even weaker computing concepts. Here we consider Turing machines with only one output state. We interpret this output status as the stopping of the Turing machine. This means that for an input \(x \in \mathbb{R}_{c} \), the Turing machine \(TM(x) \) ends its calculation after an unknown
     but finite number of arithmetic steps, or it computes forever.
    
    \begin{Definition}\label{semi}
        We call a set \(\mathcal{M} \subseteq \mathbb{R}_{c} \) semi-decidable
         if there is a Turing machine 
         \( TM_{\mathcal{M}} \) that stops for the input \( x \in \mathbb{R}_{c}\), if and only if \( x \in \mathcal{M} \) applies.
    \end{Definition}

Specker constructed in \cite{S50} a monotonically increasing computable sequence $\{r_n\}_{n\in\NN}$ of rational numbers that is bounded by 1, but the limit $x^*$, which naturally exists, is not a computable number. 
For all $M\in\NN$ there exists $n_0=n_0(M)$ such that for all $n\geq n_0$, $0\leq x-r_n <\frac 1{2^M}$ always 
holds, but the function $n_0:\NN\to\NN$ is not
partial recursive.
This means there are computable monotonically increasing sequences of rational numbers, which each converge to a finite limit value, but for which the limit values are not computable numbers and therefore the convergence is not effective.
Of course, the set of computable numbers is countable.

\begin{Definition}\label{Mazur}
    A function $f : \RR_c \to \RR_c$ is called Banach-Mazur computable if $f$ maps any given computable sequence
$\{x_n\}_{n=1}^\infty$ of real numbers into a computable sequence
$\{f(x_n)\}_{n=1}^\infty$ of real numbers.
\end{Definition}
\begin{Definition}\label{Borel}
A function $f : \RR_c\to \RR_c$ is called Borel-Turing
computable if there is an algorithm that transforms each
given computable sequence of a computable real $x$ into a
corresponding representation for $f(x)$.
\end{Definition}

\begin{Remark}
    The concatenation of Borel-Turing computable functions is again a Borel-Turing computable function.
\end{Remark}

We note that Turing’s original definition of computability conforms
to the definition of Borel-Turing computability above. Banach-Mazur  computability (see Definition~\ref{Mazur})
is the weakest form of computability. For an
overview of the logical relations between different notions
of computability we again refer to \cite{AB14}.

\subsection{Basics of channel coding and zero-error theory}

 In the theory of transmission, the receiver must be in a position to
successfully decode all the messages transmitted by the sender.

Let $\X$ be a finite alphabet. We denote the set of probability distributions by $\P(\X)$.
We define the set of computable
probability distributions $\P_c(\X)$ as the set of all probability distributions $P\in\P(X)$ such that $P(x)\in \RR_c$ for all 
$x\in\X$.
Furthermore, for finite alphabets $\X$ and $\Y$, let $\CH$ be the set of all conditional probability
distributions (or channels) $P_{Y|X} : \X \to \P(\Y)$.
$\CHc$ denotes the set of all computable conditional probability
distributions, i.e., $P_{Y|X} (\cdot|x) \in \P_c(\Y)$ 
for every $x\in\X$.

 Let $M\subset\CHc$. We call $M$ semi-decidable (see Definition~\ref{semi}) if and
    only if there is a Turing machine $TM_M$ that either stops or computes forever, depending on whether $W\in M$ is true. That means $TM_M$ accepts exactly the elements of $M$ and calculates forever for an input $W\in M^c=\CHc\setminus M$.

\begin{Definition}
	A discrete memoryless channel (DMC) 
	is a triple $(\X,\Y,W)$, where $\X$ is the finite input alphabet, 
	$\Y$ is the finite output alphabet, and 
	$W(y|x)\in\CH$ with $x\in\X$, $y\in\Y$.
	The probability for a sequence $y^n\in\Y^n$ to be received if 
	$x^n\in\X^n$ was sent is defined by
	$$
	W^n(y^n|x^n)=\prod_{j=1}^n W(y_j|x_j).
	$$
\end{Definition}

\begin{Definition}
    A block code $\C$ with rate $R$ and block length $n$ consists of 
    \begin{itemize} 
    \item A message set $\M=\{ 1,2,...,M \}$ with $M=2^{nR}\in\NN$.
    \item An encoding function $e:\M\to \X^n$.
    \item A decoding function $d:\Y^n\to\M$.
    \end{itemize}
    We call such a code an $(R,n)$-code.
\end{Definition}


\begin{Definition}\mbox{}
\begin{enumerate}
    \item 
    The individual message probability of    error is defined by the conditional probability of error given that message $m$ is transmitted: \[
    P_{m}(\C)=Pr\{d(Y^n)\neq m|X^n=e(m)\}.
    \]
    \item We define the maximal probability of error by $P_{\max}(\C)=\max_{m\in\M} P_{m}(\C)$.
    
    \item A rate $R$ is said to be achievable if there exist a sequence of $(R,n)$-codes $\{\C_n\}$ with  probability  of  error $P_{\max}(\C_n)\to 0$ as $n\to \infty$.
    \end{enumerate}
\end{Definition}

Two sequences $x^n$ and $x'^n$ of size $n$ of input variables are distinguishable by a receiver if  the vectors $W^n(\cdot|x^n)$ and $W^n(\cdot|x'^n)$ are orthogonal. That means if $W^n(y^n|x^n)>0$ then $W^n(y^n|x'^n)=0$
and if $W^n(y^n|x'^n)>0$ then $W^n(y^n|x^n)=0$.
We denote by $M(W,n)$ the maximum cardinality of a set of mutually orthogonal vectors among the $W^n(\cdot|x^n)$ with $x^n\in\X^n$.

There are different ways to define the capacity of a channel. The so-called pessimistic  capacity is defined as $\liminf_{n\to\infty} \frac {\log_2 M(W,n)}n$
and the optimistic capacity is defined as $\limsup_{n\to\infty} \frac {\log_2 M(W,n)}n$. A discussion about these quantities can be found in \cite{A06}. 
We define
the zero-error capacity of $W$ as follows.
	\[
	C_0(W) = \liminf_{n\to\infty} \frac {\log_2 M(W,n)}n
	\]
For the zero-error capacity, the pessimistic capacity and the optimistic capacity are equal.

The zero-error capacity can be characterized in graph-theoretic terms as well. Let $W\in\CH$ be given and $|\X|=q$.
	Shannon \cite{S56} introduced the confusability graph $G_W$ with $q=|G|$. In this graph, two letters/vertices $x$ and $x'$ are connected, if one could be confused
	with the other due to the channel noise (i.e. there does exist a $y$ such that $W(y|x)>0$ and $W(y|x')>0$).
	Therefore, the maximum independent set is the maximum number of single-letter messages which can be sent without danger of confusion.  
	In other words, the receiver knows whether the received message is correct or not.
		It follows that $\alpha(G)$ is the maximum number of messages which can be sent without danger of confusion.  
		Furthermore, the definition is extended to words of length $n$ by
		$\alpha(G^{\boxtimes n})$. Therefore, we can give the following graph-theoretic definition of the Shannon capacity.
\begin{Definition} 
	The Shannon capacity of a graph $G\in\G$ is defined by
	\[
	\Theta(G) \coloneqq \limsup_{n\to\infty} \alpha(G^{\boxtimes n})^{\frac 1n}.
	\]
\end{Definition}
Shannon discovered the following.
\begin{Theorem}[Shannon \cite{S56}]\label{Shannon}
Let $(\X,\Y,W)$ be a DMC. Then 
 	\[
 	2^{C_0(W)}=\Theta(G_W)=\lim_{n \to \infty} \alpha(G_W^{\boxtimes n})^{\frac 1n}.
	\]
 \end{Theorem}
 
		This limit exists and equals the supremum \[
		\Theta(G_W)=\sup_{n\in\NN} \alpha(G_W^{\boxtimes n})^{\frac 1n}\] by Fekete's lemma.

\begin{Remark}
    The zero-error capacity plays an important role in several other problems from control theory as well, see \cite{Wi19}
\end{Remark}

Observe that Theorem~\ref{Shannon} yields no further information on whether $C_0(W)$ and $\Theta(G)$ are computable real numbers.

Shannon also considered the zero-error capacity for channels
with feedback in \cite{S56},
 Let $W\in\CH$, then Shannon showed in \cite{S56}
\be
C_0^{FB}=\left\{\begin{array}{ll} 0 & \text{if}\  C_0(W)=0\\
\max_P\min_y \log_2\frac 1{\sum_{x:W(y|x)>0}P(x)} & otherwise. \end{array}\right\}.
\ee
If we set
\be
\Psi(W)=\min_{p\in\P(\X)} \min_{y\in\Y}\sum_{x:W(y|x)>0} P(x),
\ee
then we have for $W$ with $C_0(W)\neq 0$,
\[
C_0^{FB}= log_2 \frac 1 {\Psi_{FB}(W)}.
\]

\subsection{Basics of control theory}

We use noatations, problem formulations and results from \cite{MS07}.
We consider the case where the plant is affected by additive exogenous disturbances $\zeta(t)$ and the sensor is noisy.
We denote by $\RR^{n\times n}$ the set of all $n\times n$ matrices $A$ with 
values in $\RR$ and by $\RR^{n\times n}_c$ the set of all computable $n\times n$
matrices over $\RR$.

\begin{Definition}
Let $A\in\RR^{n\times n}$ and $C\in\RR^{l\times n}$, $x_0\in\RR$ be the initial state and the exogenous disturbance $\zeta(t)$ be random vectors. Then we call a system with a state function $x:\NN_0\to\RR^n$ and a measured output function $y:\NN_0\to\RR^l$ with
\be
x(t+1)=Ax(t)+\zeta(t),\ x(0)=x_0,\ y(t)=Cx(t)
\ee
an discrete-time invariant linear system.
We call the system unstable, if there is an eigenvalue $\lambda$ of $A$ with
$|\lambda|\geq 1$.
\end{Definition}

\begin{Remark}
    Throughout this work, we require the pair \((A,C)\) to satisfy \emph{detectablility}. This property relates to the question whether it is possible, in general, to estimate the state \(x(n)\) based directly on the sequence \(y(n),~ y(n-1),~ y(n-2),~\ldots\) of sensor outputs. A formal definition may be found in \cite{bookMatveev}.
\end{Remark}

The set of \(n\times n\) matrices that describe an unstable system is denoted by \(\M^{n\times n}\). Likewise, \(\M_c^{n\times n}\) denotes the set of all computable \(n\times n\) matrices that describe an unstable system.
We consider the following system (see Figure~\ref{fig:GeneralSchematics}).
The system is controlled by the plant. A sensor measures the quantity to be controlled. The goal of the remote station (decoder estimation) is to estimate the current state on the basis of the prior measurements.
The sensor can communicate information to the remote station via a DMC
$(\E,\S,W)$. A Coder translates the measurements into a sequence of the finite input alphabet $e^n\in\E^n$. Depending on the received sequence $s^n\in\S^n$ the decoder estimator produce an estimate $\hat{x}$:
\be
\hx(t)=\X(t,s(0),s(1),\dots,s(t)).
\ee
We consider the model from Figure~\ref{fig:GeneralSchematics} with and without the feedback link. The coder learns $s (t)$ from the feedback at time $t + 1$. Therefore the encoding function with feedback is 
\be
e(t)=\E(t,y(0),\dots,y(t),s(0),\dots,s(t-1))\in\E
\ee
and the encoding function without feedback is 
\be
e(t)=\E(t,y(0),\dots,y(t))\in\E.
\ee

In this model we assume that we have a noiseless feedback channel (see also \cite{S01, TM04}). Such an assumption can be made in practice if the transmission power of the receiver is much greater than that of the transmitter. A typical application is the communication between a satellite and the station on earth (see \cite{SV04}).

\begin{Definition}
The coder and decoder-estimator are said to almost surely track the state with a bounded
error if
\be
\limsup_{t\to\infty} |x(t)-\hx(t)|<\infty\ \text{almost surely}
\ee

\end{Definition}

In control theory the question arises whether there is a coder and a decoder, such that
they almost surely track the state with a bounded
error.  
Let $n\in\NN$ and $W\in\CH$ be fixed. Denote $\big(\lambda_1(A),\dots,\lambda_n(A)\big)$ the
family of eigenvalues of $A\in \M^{n\times n} $ and
\be\label{eta}
\eta(A)=\sum_{j:|\lambda_j(A)|\geq 1} \log_2|\lambda_j(A)|
\ee
The mapping \(\eta : \M^{n\times n} \rightarrow \RR\) will subsequently be essential in our considerations on the decidability of the state estimation an stabilization problem.
In this context, we also require the following lemma:

\begin{Lemma}\label{L2}
The function $\eta: \M_c^{n\times n}\to \RR_c$ is a Turing computable function,
that means there exists a Turing machine $TM_M$ such that $TM_M(A) = \eta(A)$ holds for all $A$ from the above set.
\end{Lemma}
\begin{proof}
$T(A)=AA^T$ is a Turing-computable map. Furthermore, 
\begin{eqnarray*}
\eta(T(A)) &=& \sum_{l=1: \lambda_l(AA^t)\geq 1}^n \log_2 |\lambda_l(AA^T)|\\
&=& \sum_{l=1: |\lambda_l(AA^t)|\geq 1}^n \log_2 |\lambda_l(A)|^2\\
&=& 2\eta(A).
\end{eqnarray*}
From this follows $\eta(A)=\frac 12 \eta(T(A))$. In addition, the mapping \[
Spec(T(A))=\begin{pmatrix} \lambda_1(T(A))\\ \hdots \\ \lambda_n(T(A))\end{pmatrix}
\]
is Turing computable (see \cite{PoRi17}).
Furthermore, the mapping $\phi:\RR_c\to\RR_c$ with $\phi(\xi)=\max\{1,\xi\}$ is Turing-computable. Thus the mapping 
\[
\eta(A)= \frac 12 \sum_{j:|\lambda_j(A)|\geq 1} \log_2 |\lambda_l(T(A))| = \frac 12 \sum_{l=1}^n \log_2 |\phi(\lambda_l(T(A)))| 
\]
is Turing-computable.
\end{proof}
\begin{Remark}
Note that the mapping $\log_2:(1,\infty)\to\RR_c$ is Turing-computable and the concatenation of Turing-computable functions is Turing-computable again.
\end{Remark}

\section{The state estimation problem}\label{State}

We consider the unstable invariant linear system 
\be\label{system}
x(t+1)=Ax(t)+\zeta(t),\ x(0)=x_0,\ y(t)=Cx(t)
\ee
We now consider the model from Figure~\ref{fig:SchematicsStateEstimation}.
In this scenario, the plant is affected by a random disturbance \(\zeta(n)\). The objective is the tracking of the Plan's state \(x(n)\). No control-link is available. We consider both the case with and without a feedback-link for the DMC \(W\). We make the following assumptions:

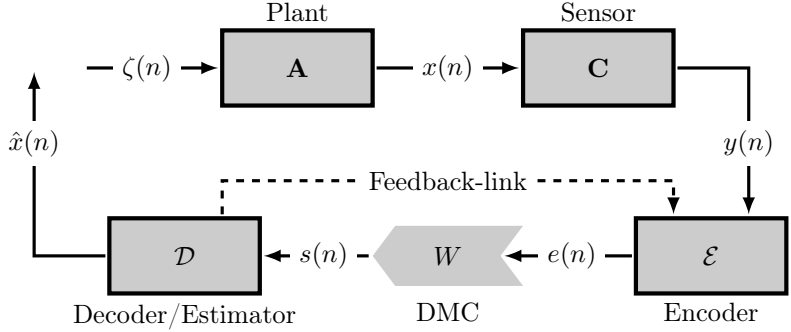
\begin{figure}[htbp!]\linespread{1}
    \centering
    \begin{tikzpicture}
            \draw[very thick, ->, shorten >=1.5pt] (-4,1) -- (-3,1); 
            \draw[very thick, ->, shorten >=1.5pt] (-4.8,1) -- (-4,1); 
            \draw (-4,1)  node[fill = white, anchor=center, align=center]
            {\(\zeta(n)\)};
            \draw[ultra thick, fill=black!20] (-3,1.5) rectangle (-1,0.5);
            \draw (-2,1)  node[anchor=center] {\(\mA\)};
            \draw (-2,1.5)  node[anchor=south] {Plant};
            \draw[very thick, ->, shorten >=1.5pt, shorten <=1.5pt] (-1,1) -- (1,1); 
            \draw (0,1)  node[fill = white, anchor=center]
            {\(x(n)\)};
            \draw[ultra thick, fill=black!20] (3,1.5) rectangle (1,0.5);
            \draw (2,1)  node[anchor=center] {\(\mC\)};
            \draw (2,1.5)  node[anchor=south] {Sensor};
            
            \draw[very thick, <-, shorten >=1.5pt, shorten <=1.5pt] (4,-1) -- (4,1) -- (3,1); 
            \draw (4,0)  node[fill = white, anchor=center, align=center]
            {\(y(n)\)};
            \draw[ultra thick, fill=black!20] (4.5,-2) rectangle (2.5,-1);
            \draw (3.5,-1.5)  node[anchor=center] {\(\E\)};
            \draw (3.5,-2)  node[anchor=north] {Encoder};
            \draw[very thick, ->, shorten >=1.5pt, shorten <=1.5pt] (2.5,-1.5) -- (0.7,-1.5); 
            \draw (1.65,-1.5)  node[fill = white, anchor=center, align=center]
            {\(e(n)\)};
            \fill[fill = black!20] (-1,-1.5) -- (-0.7, -1.1) -- (1,-1.1) -- (0.7,-1.5) -- (1,-1.9) -- (-0.7,-1.9);
            \draw (0,-1.5)  node[anchor=center] {\(W\)};
            \draw (0,-2)  node[anchor=north] {DMC};
            \draw[very thick, ->, shorten >=1.5pt, shorten <=1.5pt] (-1,-1.5) -- (-2.5,-1.5); 
            \draw (-1.65,-1.5)  node[fill = white, anchor=center, align=center]
            {\(s(n)\)};
            \draw[ultra thick, fill=black!20] (-4.5,-2) rectangle (-2.5,-1);
            \draw (-3.5,-1.5)  node[anchor=center] {\(\D\)};
            \draw (-3.5,-2)  node[anchor=north] {Decoder/Estimator};
            \draw[dashed, very thick, ->, shorten >=1.5pt, shorten <=1.5pt] (-3,-1) -- (-3,-0.5) -- (3,-0.5) -- (3,-1); 
            \draw (0,-0.5)  node[fill = white, anchor=center, align=center]
            {Feedback-link};
            \draw[very thick, ->, shorten >=1.5pt, shorten <=1.5pt] (-4.5,-1.5) -- (-5.5,-1.5) -- (-5.5,1);
            \draw (-5.5,0)  node[fill = white, anchor=center, align=center]
            {\(\hat{x}(n)\)};
    \end{tikzpicture}
    \caption{Remote state estimation.}
    \label{fig:SchematicsStateEstimation}
\end{figure}

\begin{Assumptions}\label{assumptions}
\begin{enumerate}
    \item We use a DMC $W$ as a communication channel. This means that for an input $e(t)$ the output $s(t)$ is independent of all other inputs and outputs. The transmission probability of the channel does not depend on the time.
    \item We set $s(t) = \emptyset$ if the message $e (t)$ has been lost through the channel.
    \item The system \eqref{system} does not effect the channel. This means that given an input $e(t)$ the output $s(t)$ is independent of $s_0$ and $\zeta$.
    \item $\zeta(t)$ is identical distributed according to a probability density $p(\zeta)$ mutually independent and independent of $s_0$.
    \item The pair $(A,C)$ is detectable.
\end{enumerate}
\end{Assumptions}

Under this assumptions in \cite{MS07} the following results are obtained.
\begin{Lemma}[\cite{MS07}]\label{MS1}
Let $W\in \CH$.
\begin{enumerate}
    \item If $\eta(A)>C_{0}^{FB}(W)$, then for any coder with feedback link and decoder estimator the 
estimation error is almost surely unbounded:
\[
\limsup_{t\to\infty}|x(t)-\hx(t)|=\infty\ \text{almost surely.}
\]
\item If $\eta(A)>C_{0}(W)$, then for any coder without feedback link and decoder estimator the 
estimation error is almost surely unbounded:
\[
\limsup_{t\to\infty}|x(t)-\hx(t)|=\infty\ \text{almost surely.}
\]
\end{enumerate}

\end{Lemma}

It follows from Lemma~\ref{MS1} that the inequalities $\eta(A)\leq C_{0}(W)$
and $\eta(A)\leq C_{0}^{FB}(W)$ are necessary for the existence of a coder and decoder-estimator without and with feedback that keep the estimation error bounded
with a non-zero probability.

In \cite{MS07} the next theorem concerning the sufficiency
of these conditions are given.
\begin{Lemma}[\cite{MS07}]
Let $W\in\CH$, $D\in\RR^+$ and $\zeta(t)\leq D$ for all $t$. 
\begin{enumerate}
    \item If $C_0(W)>\eta(A)$ then there exists a coder without feedback and a decoder estimator that with probability 1 track the state with a bounded error. 
    \item  If $C_{0}^{FB}(W)>\eta(A)$ then there exists a coder with feedback and a decoder estimator that with probability 1 track the state with a bounded error. 
\end{enumerate}

\end{Lemma}

In summary, through these two theorems, the following applies:
\begin{itemize}
\item Condition $\eta(A)>C_{0}(W)$ is sufficient to ensure that the state estimation problem cannot be solved for channel $W$ without feedback.
\item Condition $\eta(A)>C_{0}^{FB}(W)$ is sufficient to ensure that the state estimation problem cannot be solved with feedback for channel $W$.
\item If $C_{0}(W)>\eta(A)$ then applies, the state estimation problem without feedback for channel $W$ can be solved.
\item If $C_{0}^{FB}(W)>\eta(A)$ then applies, the state estimation problem with feedback for channel $W$ can be solved.
\end{itemize}


In the following we consider the algorithmic solvability of the state estimation problem with fixed dimension $n$ at the plant.

\subsection{Problem formulation}
Our Turing machines for the question have as inputs 
\[
\{(A, W)\}\in\M_c^{n\times n}\times \CHc. 
\]

For $n\in\NN$ we consider the following sets:
\begin{align*}
&S_{SE}(solvable) \\
&\qquad = \{(A,W)\in \M_c^{n\times n}\times \CHc: C_0(W)>\eta(A)\}\\
&S_{SE}(unsolvable) \\
&\qquad = \{(A,W)\in \M_c^{n\times n}\times \CHc: C_0(W)<\eta(A)\}\\
&S_{SE}^{FB}(solvable) \\
&\qquad = \{(A,W)\in \M_c^{n\times n}\times \CHc: C_0^{FB}(W)>\eta(A)\}\\
&S_{SE}^{FB}(unsolvable) \\
&\qquad = \{(A,W)\in \M_c^{n\times n}\times \CHc: C_0^{FB}(W)<\eta(A)\}\\
\end{align*}

We now ask: Do Turing machines $TM_1,\dots,TM_4$ with only one stop state exist for the sets 
$S_{SE}(solvable),S_{SE}(unsolvable), S_{SE}^{FB}(solvable),S_{SE}^{FB}(unsolvable)$, respectively, so that, for example, $TM_1$ stops for the input $(A,W)\in \M_c^{n\times n}\times \CHc$ exactly when $(A,W)\in S_{SE}(solvable)$ holds, that is, the Turing machine $TM_1$ is supposed to accept the inputs of the set $S_{SE}(solvable)$ exactly.

Furthermore, these Turing machines only have one holding state "stop" and should stop if and only if $(A, W)$ are in the respective interesting sets. Then the relevant set is accepted by the corresponding Turing machine.
\begin{Theorem}\label{T1}
The set 
\[
S_{SE}(solvable)\subset \M_c^{n\times n}\times\CHc 
\]
is not semi-decidable for channels without feedback.
\end{Theorem}
\begin{proof}
Suppose there is a Turing machine $TM_{SE}$ with the required properties. Let 
$\mu\in\RR_c^+$, $\lambda=2^\mu$ (that means $\lambda\in\RR_c$) and let 
\[
A_\lambda=\begin{pmatrix} \lambda_1 & 0 & \dots & 0 & 0 \\ 0 & \lambda_2 & 0 & \dots  & 0\\&&\hdots \\
0 & \dots  & 0 &0  & \lambda_n \end{pmatrix}\in\M^{n\times n}_c
\]
Therefore $\eta(A_\lambda) = \mu$. 
Now $(A_\lambda,W)\in S_{SE}(solvable)$ holds if and only if $\eta(A_\lambda)<C_0(W)$ holds.
Hence $\mu<C_0(W)$. We fix this $\mu$ and consider the set \[
\M(\mu)=\{W\in\CHc : C_0(W)>\mu\}.
\]
The Turing machine $TM_{SE}$ stops for input $(A_\lambda, W)$ if and only if $W\in \M(\mu)$ holds. The set $\M(\mu)$ is then semi-decidable. But with this we have reached a contradiction to our result in \cite{BD20Z}. Here we have shown that for all 
$0<\mu<\log_2(\min\{|\X|,|\Y|\})$ the set $\M(\mu)$ is not semi-decidable. With that we have proven the theorem.
\end{proof}
For the next result we need a tightening of a result of \cite{BD20R}. We prove the following:
\begin{Lemma}\label{sharp}
There is a Turing machine $TM_*$ with inputs $\lambda\in\RR_c^+$ and 
$W\in\CHc$ with the following properties: The Turing machine $TM_*$ has only one output state "stop" and otherwise computes forever. The Turing machine $TM_*$ stops for the input $(\lambda, W)$ if and only if the relationship 
$C_0^{FB}<\lambda$ holds.
\end{Lemma}
\begin{proof}
In \cite{BD20R} a Turing machine $TM^{FB}$ was constructed that receives the inputs $N\in\NN$ and $W\in\CHc$ and computes a function 
$F_N(W) = TM^{FB}(W,N)$ with $F_N(W)<F_{N+1}(W)$ and 
$\lim_{N\to\infty} F_N(W)=C_0^{FB}(W)$. Further, we use a Turing machine 
$TM_<$ that receives the inputs $\xi,\lambda\in\RR_c$. 
$TM_<$ stops for the input $(\xi, \lambda)$ if and only if $\xi<\lambda$ applies. We can trace this back to a Turing machine $TM_{<0}$ which decides for input $\xi\in\RR_c$ whether $xi <0$ and stops if and only if $xi <0$ applies.
We now construct $TM_*$. This Turing machine receives the inputs $\lambda\in\RR_c^+$ and $W\in\CHc$. We construct $TM_*$ as follows. For 
$N = 1$ we calculate $\Psi_1 (W)$ and start the Turing machine 
$TM_{<0}$ for the input $\Psi_1(W)-\lambda$ for this computable number we have a program that computes this number. $TM_{<0}$ receives this number as an input. A computation step is carried out for $TM_{<0}$ when this Turing machine stops. This means $\Psi_1(W)-\lambda <0$ and thus 
$C_0^{FB}-\lambda <0$.
If $TM_{<0}$ has not stopped for input $\Psi(W)-\lambda$, then we compute $\Psi_2(W)-\lambda$ in the next loop and carry out the first computation step for $\Psi_2-\lambda$. If one of the two loops then stops, then, as already shown, $C_0^{FB}(W)-\lambda <0$, etc.
It is clear that the Turing machine $TM_*$ stops if and only if there is an $N_0\in\NN$ such that $\Psi_{N_0}<0$. But this is true if and only if 
$C_0^{FB} <0$. So we have a Turing machine $TM_*$ with the desired properties. The central statement is that $TM_*$ is recursively dependent on $W$ and $\lambda$.
Now we show that the Turing machine $TM_*$ already has the desired properties. Let $(A, W)\in\M^{n\times n}_c\times\CHc$. 
Then $\eta(A)=TM_\eta(A)\in\RR_c$ and $TM_\eta$ computes from the program for describing $A$ a program for computing the number $\eta(A)$. With this we compute $\eta(A)$ and use the Turing machine $TM_*$ for $(W,\eta(A))$. This stops exactly when $C_0^{FB}(W)-\eta(A)<0$. We have thus proven the desired property.
\end{proof}

We continue to show.
\begin{Theorem}\label{T2}
The set 
\[
S_{SE}^{FB}(solvable)\subset \M_c^{n\times n}\times\CHc 
\]
is not semi-decidable for channels with feedback.
\end{Theorem}
\begin{proof}
We use the proof idea from the proof of Theorem~\ref{T1} and show: If the above set $S_{SE}^{FB}$ is semi-decidable, then for all $\mu$ with $0 <\mu<\log(\min\{|\X|,|\Y|\})$ the set 
\be
\{W\in\CHc : \mu<C_0^{FB}(W)\}
\ee
would be semi-decidable.
We showed in \cite{BD20R} that for all $0 <\mu<\log_2(\min\{|\X|,|\Y|\})$ this set is not semi-decidable. With this we have a contradiction and proven the theorem.
\end{proof}


We now want to examine the unsolvability of the state estimation problem. We first consider the case of channels with feedback.

\begin{Theorem}\label{T3}
The set 
\[
S_{SE}^{FB}(non-solvable)\subset \M^{n\times n}_c \times \CHc 
\]
is semi-decidable for channels with feedback.

\end{Theorem}
\begin{proof}
Let $A\in\M^{n\times n}_c $ and $W\in\CHc$ then 
\[
(A,W)\in S_{Se}^{FB}(non-solvable)\Leftrightarrow \eta(A)>C_0^{FB}(W).
\]
The mapping $\eta:\M^{n\times n}_c \to \RR_c$ is Turing-computable. 
Thus $(A,W)\in S^{FB}_{SE}(non-solvable)$ is fulfilled if and only if 
$W\in\{ W\in\CHc : C_0^{FB}(W)<\eta(A)\}$ is fulfilled. 
We now construct a Turing machine $TM_*$ which receives $\mu\in\RR_c^+$ and 
$W\in\CHc$ as input. $TM_*$ only has the initial state "stop" and otherwise computes forever and stops if and only if $C_0^{FB}(W)<\mu$ applies to the input $(\mu,W)$.
To do this, we sharpen the semi-decidability problem (see \cite{BD20R})
in Lemma~\ref{sharp}
for the set 
\be
\{W\in\CHc : 0<C_0^{FB}(W)<\lambda\} 
\ee
with $\lambda \in \RR_c^+$, where the program that solves the Turing machine recursively depends on $\lambda$. The Turing machine $TM_*$ 
from Lemma~\ref{sharp} delivers this result immediately.
\end{proof}
\begin{Remark}
We do not know whether a result similar to Theorem~\ref{T3} also applies to channels without feedback. 
We know that for all $\lambda\in\RR_c^+$ the set 
\be\label{set1}
\{W\in\CHc : C_0^{FB}(W)<\lambda\} 
\ee
is semi-decidable, but for the set 
\be\label{set2}
\{W\in\CHc : C_0(W)<\lambda\} 
\ee
this is still open for $\lambda\in\RR_c^+$. 
This question is coupled with other interesting questions from information theory and combinatorics. The question of the semi-decidability of the set
\eqref{set1} for $\lambda\in\RR_c^+$ is equivalent to the question of the semi-decidability of the set 
\be\label{set3}
\{G:\G:\Theta(G)>2^\lambda\},
\ee
for the zero-error capacity of graphs.

If \eqref{set3} were semi-decidable for all $\lambda\in\RR_c^+$, then the Turing computability could be shown by the Shannon zero-error capacity of graphs (see also the discussion in \cite{BD20Z}).

\end{Remark}

\subsection{Fixed unstable System or fixed Channel}

We have so far considered the scenario that a Turing machine $(A,W)$ receives as input. For some practical applications it is also interesting that $A$ is fixed and only the DMC varies. This means for example that the unstable system $A$ is fixed and the channel varies. This occurs practically as for example always in wireless communication and also in optical communication by changing the physical parameters of the optical fiber due to temperature variations. The question is now whether for the fixed unstable system $A$ a Turing machine exists, which stops for the input channel $W\in\CHc$ exactly then, if according to Figure~\ref{fig:GeneralSchematics} over this channel the state $\hx$ can be computed according to the strong requirement. 
First we consider the model without the feedback link.
This corresponds to the question whether the set 
\[
S_{>\eta}(A):=\{W\in\CHc : C_0(W)>\eta(A)\}
\] is semi-decidable.
An answer is given by the following theorem.
\begin{Theorem}
Let $\X,\Y$ be fixed alphabets and $A\in \M_c^{n\times n}$ with $0<\eta(A)<\log_2(\min\{|\X|,|\Y|\})$. Then the set $S_{>\eta}(A)$ is not semi-decidable. If $\eta(A)=\log_2(\min\{|\X|,|\Y|\}$, then the set $S_{>\eta}(A)$  is empty.
\end{Theorem}
\begin{proof}
Under the condition $\eta$, non-semi-decidability follows directly from the previous results. Furthermore, 
\[
\max_{W\in\CHc} C_0(W) =\log_2 (\min\{|\X|,|\Y|\})
\]
holds. Therefore for $\eta(A)=\log_2 (\min\{|\X|,|\Y|\})$ the set $S_{>\eta}(A)$ is empty.
\end{proof}

Now we consider Figure~\ref{fig:GeneralSchematics} with feedback link.
This corresponds to the question whether the set 
\[
S^{FB}_{>\eta}(A):=\{W\in\CHc : C^{FB}_0(W)>\eta(A)\}
\] is semi-decidable.
Using the same method, we obtain the following result for state estimation with feedback and fixed unstable system $A$.
\begin{Theorem}
Let $\X,\Y$ be fixed alphabets and $A\in \M_c^{n\times n}$ with $0<\eta(A)<\log_2(\min\{|\X|,|\Y|\})$. Then the set $S^{FB}_{>\eta}(A)$ is not semi-decidable. If $\eta(A)=\log_2(\min\{|\X|,|\Y|\})$, then the set $S^{FB}_{>\eta}(A)$  is empty.
\end{Theorem}

So you get the same answer for the system with and without feedback link. 

The situation is different if we consider the sets 
\[ S_{<\eta}(A):=\{W\in\CHc : C_0(W)<\eta(A)\}
\] and 
\[ S^{FB}_{<\eta}(A):=\{W\in\CHc : C^{FB}_0(W)<\eta(A)\}.
\] 
Without feedback link the question is still open! In contrast, it follows immediately from the previous results that $S^{FB}_{<\eta}(A)$ is semi-decidable.

{\bf Interesting observation:} conversely, if the channel $W\in\CHc$ is fixed and we allow the system $A$ to vary (due to fluctuations on the system), then the corresponding answers to the corresponding questions are positive (as before) when $C_0(W)\in\RR_c$. There we just have to use the Turing machine corresponding to \cite{PR17}. $\zeta=C_0(W)$ is fixed and we have a Turing machine TM which stops exactly when $\eta(A)>\zeta$. We get exactly the same for $\eta(A)<\zeta$ when $C_0(W)>0$. The same is true for fixed $W$ and $C_o^{FB}(W)\in\RR_c$. Here $C_0^{FB}(W)\in\RR_c$ is always satisfied for $W\in\CHc$.

{\bf Discussion:} We do not know whether $C_0(W)\in\RR_c$ always holds for $W\in\CHc$. If $C_0(W_*)\not\in\RR_c$ holds for a given $W_*\in\CHc$, then the above statement fails. 

\begin{Theorem}
If there exists a channel $W_*\in\CHc$ with
$C_0(W_*)\not\in\RR_c$, then the set
\[
S_0:=\{A\in \M_c^{n\times n} :\eta(A)>C_0(W_*)\}
\]
is not semi-decidable.
\end{Theorem}

\begin{proof}
We know that by assumption $C_0(W_*)\not\in\RR$ holds. $C_0(W_*)$ is a limit of a monotonically increasing computable sequence of computable numbers. Now if the set $S_0$ is semi-decidable after all, then we can find (see \cite{BM21}) a computable sequence of monotonically decreasing computable numbers converging to $C_0(W_*)$. This is a contradiction, to the assumption $C_0(W_*)\not\in\RR$.
\end{proof}

\section{The stabilization problem}\label{Stab}

We now consider the following stabilization problem (see also \cite{MS07} section~6). Here we assume again that the plan is unstable and that the assumptions~\ref{assumptions} apply. The plant is now driven by controls $u(t)\in\RR^m$:
\begin{eqnarray}\nonumber
x(t+1)&=&Ax(t)+Bu(t)+\zeta(t);\\
x(0) &=& x_0,\ \ \ y(t)=Cx(t).\label{stabilization}
\end{eqnarray}
The goal is to stabilize the system in figure~\ref{fig:SchematicsStabilization}. 
\begin{figure}[htbp!]\linespread{1}
    \centering
    \begin{tikzpicture}
            \draw[very thick, ->, shorten >=1.5pt] (-4,1.25) -- (-3,1.25); 
            \draw[very thick, ->, shorten >=1.5pt] (-4.8,1.25) -- (-4,1.25); 
            \draw (-4,1.25)  node[fill = white, anchor=center, align=center]
            {\(\zeta(n)\)};
            \draw[ultra thick, fill=black!20] (-3,1.5) rectangle (-1,0.5);
            \draw (-2,1)  node[anchor=center] {\(\mA\)};
            \draw (-2,1.5)  node[anchor=south] {Plant};
            \draw[very thick, ->, shorten >=1.5pt, shorten <=1.5pt] (-1,1) -- (1,1); 
            \draw (0,1)  node[fill = white, anchor=center]
            {\(x(n)\)};
            \draw[ultra thick, fill=black!20] (3,1.5) rectangle (1,0.5);
            \draw (2,1)  node[anchor=center] {\(\mC\)};
            \draw (2,1.5)  node[anchor=south] {Sensor};
            
            \draw[very thick, <-, shorten >=1.5pt, shorten <=1.5pt] (4,-1) -- (4,1) -- (3,1); 
            \draw (4,0)  node[fill = white, anchor=center, align=center]
            {\(y(n)\)};
            \draw[ultra thick, fill=black!20] (4.5,-2) rectangle (2.5,-1);
            \draw (3.5,-1.5)  node[anchor=center] {\(\E\)};
            \draw (3.5,-2)  node[anchor=north] {Encoder};
            \draw[very thick, ->, shorten >=1.5pt, shorten <=1.5pt] (2.5,-1.5) -- (0.7,-1.5); 
            \draw (1.65,-1.5)  node[fill = white, anchor=center, align=center]
            {\(e(n)\)};
            \fill[fill = black!20] (-1,-1.5) -- (-0.7, -1.1) -- (1,-1.1) -- (0.7,-1.5) -- (1,-1.9) -- (-0.7,-1.9);
            \draw (0,-1.5)  node[anchor=center] {\(W\)};
            \draw (0,-2)  node[anchor=north] {DMC};
            \draw[very thick, ->, shorten >=1.5pt, shorten <=1.5pt] (-1,-1.5) -- (-2.5,-1.5); 
            \draw (-1.65,-1.5)  node[fill = white, anchor=center, align=center]
            {\(s(n)\)};
            \draw[ultra thick, fill=black!20] (-4.5,-2) rectangle (-2.5,-1);
            \draw (-3.5,-1.5)  node[anchor=center] {\(\D\)};
            \draw (-3.5,-2)  node[anchor=north] {Decoder/Estimator};
            \draw[very thick, ->, shorten >=1.5pt, shorten <=1.5pt] (-4,-1) -- (-4,0.75) -- (-3,0.75); 
            \draw (-4,0)  node[fill = white, anchor=center, align=center]
            {Control-\\link};
            \draw[dashed, very thick, ->, shorten >=1.5pt, shorten <=1.5pt] (-3,-1) -- (-3,-0.5) -- (3,-0.5) -- (3,-1); 
            \draw (0,-0.5)  node[fill = white, anchor=center, align=center]
            {Feedback-link};
    \end{tikzpicture}
    \caption{Remote stabilization for disturbed plants, with and without feedback-link.}
    \label{fig:SchematicsStabilization}
\end{figure}
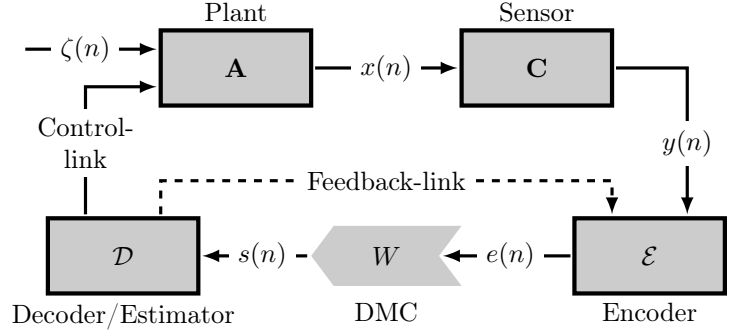

We are investigating a remote control setup. The coder emits a signal $e (t)$ into the channel. Depending on the received message $s(\theta)$ at time $t$, the decoder selects a current control: 
\be
u(t)=\U(t,s(0),s(1),\dots, s(t))
\ee

\begin{Definition}
We say that the coder and the decoder stabilize the noisy plant in \eqref{stabilization}
if
\be
\limsup_{t\to\infty} |x(t)|<\infty
\ee

\end{Definition}

Matveev and Savkin showed in \cite{MS07} the following results:
\begin{Lemma}[\cite{MS07}]\label{MS1a}
If $\eta(A)> C_0^{FB}(W)$ the stabilization error is almost sure unbounded
for any coder decoder pair:
\be
\limsup_{t\to\infty}|x(t)|=\infty\ \ \text{almost surely.}
\ee
\end{Lemma}

Because of Lemma~\ref{MS1} the inequality $\eta(A)\leq C_0^{FB}(W)$ is
necessary for existence of a coder decoder pair that stabilizes the system 
with a non-zero probability.
For the converse Matveev and Savkin showed the following:
\begin{Lemma}[\cite{MS07}]
If $\eta(A)<C_0^{FB}(W)$, the pairs $(A,C)$ and $(A,B)$ are detectabel and stabilizable,
respectively, and for a given $D$ $|\zeta(t)|\leq D<\infty$, then a properly designed
coder-decoder pair stabilzes the plant almost surely.
\end{Lemma}
With the help of these results from Matveev and Savkin, we get the following:
\begin{Theorem}\label{T4}
The set 
\[
S^{FB}_{Stability}(possible)\subset \M^{n\times n}_c\times\CHc 
\]
is not semi-decidable for channels with feedback.
\end{Theorem}
\begin{proof}
The proof works like the proof for Theorem~\ref{T2}.
\end{proof}
Furthermore, we can show the following:
\begin{Theorem}\label{T5}
The set 
\[
S^{FB}_{Stability}(impossible)\subset \M^{n\times n}_c\times\CHc 
\]
is semi-decidable for channels with feedback.
\end{Theorem}
\begin{proof}
    Again, the proof works along the same lines as the proof for Theorem \ref{T3}.
\end{proof}

\section{Remote State Estimation and Stabilization for Undisturbed Plants}\label{Weak}
It is interesting that in the practically very interesting stability and state estimation quality conditions, the zero-error capacity plays a decisive role. Next we want to examine what happens when we demand weaker conditions regarding the quality of the stability and state estimation.

\begin{figure}[htbp!]\linespread{1}
    \centering
    \begin{tikzpicture}
            \draw[ultra thick, fill=black!20] (-3,1.5) rectangle (-1,0.5);
            \draw (-2,1)  node[anchor=center] {\(\mA\)};
            \draw (-2,1.5)  node[anchor=south] {Plant};
            \draw[very thick, ->, shorten >=1.5pt, shorten <=1.5pt] (-1,1) -- (1,1); 
            \draw (0,1)  node[fill = white, anchor=center]
            {\(x(n)\)};
            \draw[ultra thick, fill=black!20] (3,1.5) rectangle (1,0.5);
            \draw (2,1)  node[anchor=center] {\(\mC\)};
            \draw (2,1.5)  node[anchor=south] {Sensor};
            
            \draw[very thick, <-, shorten >=1.5pt, shorten <=1.5pt] (4,-1) -- (4,1) -- (3,1); 
            \draw (4,0)  node[fill = white, anchor=center, align=center]
            {\(y(n)\)};
            \draw[ultra thick, fill=black!20] (4.5,-2) rectangle (2.5,-1);
            \draw (3.5,-1.5)  node[anchor=center] {\(\E\)};
            \draw (3.5,-2)  node[anchor=north] {Encoder};
            \draw[very thick, ->, shorten >=1.5pt, shorten <=1.5pt] (2.5,-1.5) -- (0.7,-1.5); 
            \draw (1.65,-1.5)  node[fill = white, anchor=center, align=center]
            {\(e(n)\)};
            \fill[fill = black!20] (-1,-1.5) -- (-0.7, -1.1) -- (1,-1.1) -- (0.7,-1.5) -- (1,-1.9) -- (-0.7,-1.9);
            \draw (0,-1.5)  node[anchor=center] {\(W\)};
            \draw (0,-2)  node[anchor=north] {DMC};
            \draw[very thick, ->, shorten >=1.5pt, shorten <=1.5pt] (-1,-1.5) -- (-2.5,-1.5); 
            \draw (-1.65,-1.5)  node[fill = white, anchor=center, align=center]
            {\(s(n)\)};
            \draw[ultra thick, fill=black!20] (-4.5,-2) rectangle (-2.5,-1);
            \draw (-3.5,-1.5)  node[anchor=center] {\(\D\)};
            \draw (-3.5,-2)  node[anchor=north] {Decoder/Estimator};
            \draw[dashed, very thick, ->, shorten >=1.5pt, shorten <=1.5pt] (-4,-1) -- (-4,1) -- (-3,1); 
            \draw (-4,0)  node[fill = white, anchor=center, align=center]
            {Control-\\link};
            \draw[dashed, very thick, ->, shorten >=1.5pt, shorten <=1.5pt] (-3,-1) -- (-3,-0.5) -- (3,-0.5) -- (3,-1); 
            \draw (0,-0.5)  node[fill = white, anchor=center, align=center]
            {Feedback-link};
            \draw[dashed, very thick, ->, shorten >=1.5pt, shorten <=1.5pt] (-4.5,-1.5) -- (-5.5,-1.5) -- (-5.5,1);
            \draw (-5.5,0)  node[fill = white, anchor=center, align=center]
            {\(\hat{x}(n)\)};
    \end{tikzpicture}
    \caption{Remote state estimation and stabilization for undisturbed plants, with and without feedback-link}
    \label{fig:SchematicsUndisturbedPlant}
\end{figure}
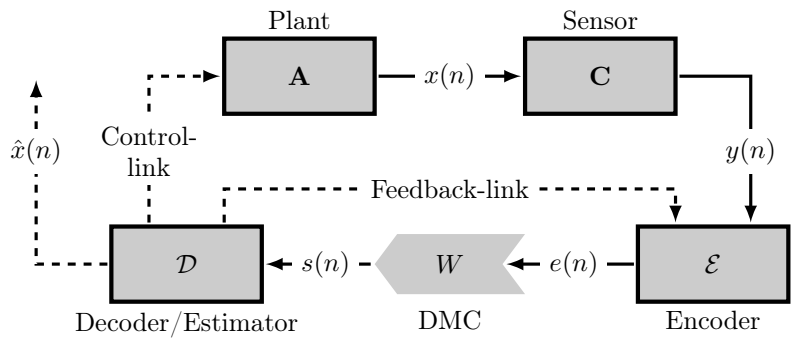

\begin{Remark}
For the not so strict error model for state estimation and stabilization, all questions can be solved algorithmically because only the classic Shannon capacity is of importance there.
\end{Remark}

We use the error model as in \cite{MS07SC}. We introduce the notations and results that we need from \cite{MS07SC}.

The consideration of noise-free plants is motivated by the objective of \cite{MS07SC}. In that paper the role of the Shannon capacity is highlighted. 
In the presence
of additive bounded disturbances, the border between the cases where the plant can
and cannot, respectively, be observed/stabilized with an almost sure bounded error is
constituted not by the ordinary $C(W)$ but the zero error capacity $C_0(W)$ of the channel, another
fundamental characteristic introduced by Shannon [55]. 
In particular, if $C_0 < \eta$ (see previous section), the
system affected by uniformly and arbitrarily small external disturbances can never be
observed/stabilized: the error is unbounded almost surely, irrespective of which causal
algorithm of observation/stabilization is employed. 

In this section, we pursue a more realistic objective of
detecting the unstable modes of the system and accept that an observer succeeds if
\be\label{2.5}
|x(t) -\hx(t)|\to 0\ as\ t\to \infty.
\ee
\begin{Definition} 
The coder-decoder pair is said to detect or track the state whenever
\eqref{2.5} is true and to keep the estimation error (or time-average error) bounded if
the following much weaker properties hold, respectively:
\be\label{2.6}
\limsup_{t\to\infty} |x(t) -\hx(t)| < \infty,\ \  \limsup_{t\to\infty} 
\frac 1t \sum_{\theta=0}^{t-1} |x(\theta) -\hx (\theta)| < \infty.
\ee
\end{Definition}

We consider Figure~\ref{fig:SchematicsUndisturbedPlant} and the model with and without feedback link.
The feedback enables the coder to be aware of the actions of
the decoder via duplicating the computations. This gives the
coder the ground to try to compensate for the previous channel errors. However, it
should be noted that Shannon showed in \cite{S56} that this feedback does not increase the rate at which the information
can be transmitted across the channel with as small a probability of error as desired. At the same time, it may increase the rate at which information can be
transmitted with the zero probability of error \cite{S56}. The feedback may also increase
the reliability function and simplify coding and decoding operations. All this can be found in the survey \cite{V98}.
The role of communication feedback in control and state estimation was discussed in
\cite{TM04, TM04b}.
The information received by the decoder is limited to a finite number of bits at
any time. So the decoder is hardly able to restore the state with the infinite exactness
$\hx(t) = x(t)$ for a finite time. In \cite{MS07SC}, the authors pursue a more realistic objective of
detecting the unstable modes of the system and accept that an observer succeeds if
\eqref{2.5} is fulfilled.
We consider the unstable discrete-time invariant linear plants of the form
\be\label{2.1}
x(t+1)=Ax(t);\ \ x(0)=x_0,\ \ y(t)=Cx(t)
\ee
The decoder is defined by an equation of the form
\be\label{2.2}
\hx(t)=\X(t,s(0),s(1),\dots,s(t)).
\ee
The coder with feedback is given by the following equation:
\be\label{2.3}
e(t)=\E(t,y(0),\dots,y(t),s(0),\dots,s(t-1))
\ee
The coder without feedback is given by the following equation:
\be\label{2.4}
e(t)=\E(t,y(0),\dots,y(t))
\ee
In this section we accept that an observer succeeds (see also \eqref{2.5}), if
\be
|x(t)-\hx(t)|\to 0\ as\ t\to\infty
\ee

\begin{Definition}\label{D3.1}
A coder-decoder pair is said to stabilize the system if
\be\label{3.3}
|x(t)| \to 0\ \ \text{and}\ \  |u(t)|\to 0\ as\ t\to\infty
\ee
and to keep the stabilization error (or time-average error) bounded if the much weaker
properties hold, respectively:
\be\label{3.4} 
\limsup_{t\to\infty} |x(t)| < \infty,\ \ \limsup_{t\to\infty} \frac 1t
\sum_{\theta =0}^{t-1} |x(\theta)| < \infty.
\ee
\end{Definition}
The problem is to bound the data rate of the channel above which there exists a
stabilizing coder-decoder pair. In this case ``stabilizing'' means either 
``stabilizing almost surely'' or ``stabilizing with as large a probability as desired.''

The following assumptions are made in this section:
\begin{Assumptions}\label{assumptions2}
\begin{enumerate}
    \item The coder sends signals to the decoder over a given stationary
discrete noisy memoryless channel \cite{S48}. 
\item The plant does not affect the operation of the channel: given
an input $e(t)$, the output $s(t)$ is statistically independent of the initial state $x_0$.
\item  The initial state $x_0$ has a probability density $p_0(x)$.
\item  The pair $(A,C)$ is detectable.
When dealing with the stabilization problem, we impose one more assumption.
\item The pair $(A,B)$ is stabilizable.
\end{enumerate}
\end{Assumptions}

For the detection problem
in \cite{MS07SC} the authors give the following:
\begin{Lemma}
Suppose that Assumptions~\ref{assumptions2} hold. Denote by $\lambda_1,\dots, \lambda_n$ the
eigenvalues of the system \eqref{2.1} repeating in accordance with their algebraic multiplicities and by $C(W)$ the capacity of the DMC. 
For the following 4 statements:
\begin{itemize}
    \item[A1] There exists a coder-decoder pair with feedback \eqref{2.2}, \eqref{2.3} that detects the state (\eqref{2.5} holds) almost surely.
    \item[A2] For arbitrary probability value $p\in(0,1)$, there exists a coder-decoder pair without feedback \eqref{2.2}, \eqref{2.4} that detects the state with the probability $p$ or better.
    \item[B] For any probability value $p\in (0,1)$, there exists a coder decoder pair with a feedback \eqref{2.2}, \eqref{2.3} that keeps the estimation error bounded with probability $p$ or better.
    \item[C] There exists a coder-decoder pair with a feedback \eqref{2.2}, \eqref{2.3} that keeps the
estimation time-average error bounded with a nonzero probability.
\end{itemize}
the following implications apply:
\[
C(W)>\eta(A)\Rightarrow A1 \Rightarrow B \Rightarrow C \Rightarrow C(W)\geq \eta(A)
\]
and 
\[
C(W)>\eta(A)\Rightarrow A2 \Rightarrow B \Rightarrow C \Rightarrow C(W)\geq \eta(A)
\]
\end{Lemma}

Now we consider the stabilization problem with the controlled version of the unstable plant:
\be\label{3.1}
x(t+1)=Ax(t)+Bu(t);\ \ x(0)=x_0,\ \ y(t)=Cx(t)
\ee
The decoder selects a control $u(t)$:
\be\label{3.2}
u(t)=\U(t,s(0),s(1),\dots,s(t)).
\ee
The coder stabilize the system if (see Definition~\ref{D3.1}):
\be
|x(t)|\to 0\ \ \text{and}\ \ |u(t)|\to 0\ \ as\ \ t\to\infty
\ee
The stabilization error is bounded, if 
\be\label{S3.4}
\limsup_{t\to\infty} |x(t)| < \infty,\ \ \limsup_{t\to\infty} \frac 1t
\sum_{\theta =0}^{t-1} |x(\theta)| < \infty.
\ee

For this weaker stabilization problem it is shown in \cite{MS07SC}:
\begin{Lemma}[\cite{MS07SC}]
Suppose that Assumptions~\ref{assumptions2} hold. Denote by $\lambda_1,\dots, \lambda_n$ the
eigenvalues of the system \eqref{2.1} repeating in accordance with their algebraic multiplicities and by $C(W)$ the capacity of the DMC. 
For the following 4 statements:
\begin{itemize}
    \item[A] There exists a coder without a communication feedback \eqref{2.4} and a decoder \eqref{3.2} that stabilize the system almost surely.
    \item[B] There exists a coder with a communication feedback \eqref{2.3} and a decoder \eqref{3.2} that stabilize the system almost surely.
    \item[C] For arbitrarily probability value $p \in (0, 1)$, there exists a coder with a
communication feedback \eqref{2.3} and a decoder \eqref{3.2} that stabilize the system with
the probability $p$ or better.
    \item[D] There exists a coder with a communication feedback \eqref{2.3} and a decoder \eqref{3.2}
that keep the time-average stabilization error bounded with a nonzero probability.
\end{itemize}
the following implication applies:
\[
C(W)>\eta(A)\Rightarrow A \Rightarrow B \Rightarrow C \Rightarrow D \Rightarrow C(W)\geq \eta(A)
\]
\end{Lemma}

Based on these results from \cite{MS07} we consider the sets 
\[
S_>=\{(A,W)\in\M^{n\times n}_c\times \CHc : \eta(A)-C(W)>0\}
\]
and 
\[
S_<=\{(A,W)\in\M^{n\times n}_c\times \CHc : \eta(A)-C(W)<0\}
\]
For these sets, the following applies:
\begin{Theorem}
Let $\X,\Y$ be finite alphabets and $n$ fixed, then the sets
$S_>$ and $S_<$ are semi-decidable.
\end{Theorem}
\begin{proof}
Lemma~\ref{L2} shows that $\eta$ is Turing computable. Because of \cite{BSBP19} the 
Shannon capacity $C$ is also Turing computable. The sets $\{x\in\RR_c: x<0\}$
and $\{x\in\RR_c: x>0\}$ are semi-decidable.
We now want to show the semi-decidability of $S_>$. Let $(A, W)$ be corresponding inputs. We calculate $\eta(A)$ and $C(W)$ with the Turing machines $TM_\eta$ and $TM_C$. These provide the computable numbers $\eta(A)$ and $C(W)$, that is, there is an algorithm for $TM_\eta$ and $TM_C$ to compute these numbers. These are input for the Turing machine $TM_{>0}$, which stops at the input $\xi=\eta(A)-C(W)$ (this is a computable number and we have an algorithm for it) exactly when $\xi>0$ applies. With this we have proven the statement of the theorem for $S_>$. The proof for $S_<$ is the same.
\end{proof}

\section{Conclusions}\label{Conclusions}
The result we gave in this paper are applications of the techniques and theories of the work on the reliability function \cite{BD20R} and on the zero-error capacity for noisy channels \cite{BD20Z} that highlight the boundaries of computer-aided design and autonomous systems. The occurrence of zero-error capacity with and without feedback in control theory is very exciting. 

\section*{Acknowledgments}

Many thanks to Moritz Wiese, who drew our attention to the works of Matveev and Savkin.

We thank the German Research Foundation (DFG) within the Gottfried Wilhelm Leibniz Prize under Grant BO 1734/20-1 for their support of H. Boche.

Further, we thank the German Research Foundation (DFG) within Germany’s Excellence Strategy EXC-2111—390814868 for their support of H. Boche.

Thanks also go to the German Federal Ministry of Education and Research (BMBF) within the national initiative for “Post Shannon Communication (NewCom)” with the project “Basics, simulation and demonstration for new communication models” under Grant 16KIS1003K for their support of H. Boche and with the project “Coding theory and coding methods for new communication models” under Grant 16KIS1005 for their support of C. Deppe.

\section*{References}

\bibliography{references}

\end{document}